\documentclass[reqno,12pt]{amsart}

\usepackage{color} 
\usepackage{ifpdf}
\ifpdf
    \usepackage[pdftex]{graphicx}
    \usepackage[pdftex]{hyperref}
    \hypersetup{
        unicode=false,          
        pdftoolbar=true,        
        pdfmenubar=true,        
        pdffitwindow=false,     
        pdfstartview={FitH},    
        pdftitle={MCP Article},      
        pdfauthor={Michael Holst},   
        pdfsubject={Mathematics},    
        pdfcreator={Michael Holst},  
        pdfproducer={Michael Holst}, 
        pdfkeywords={PDE, analysis, mathematical physics}, 
        pdfnewwindow=true,      
        colorlinks=true,        
        linkcolor=red,          
        citecolor=blue,         
        filecolor=magenta,      
        urlcolor=cyan           
    }

\else
    \usepackage{graphicx}
    \usepackage{pstricks}
    
    \newcommand{\href}[2]{#2}
\fi

\usepackage{times}
\usepackage{amsfonts}

\usepackage{amsmath}
\usepackage{amsthm}
\usepackage{amssymb}
\usepackage{amsbsy}
\usepackage{amscd}

\usepackage{enumerate}
\usepackage{verbatim}
\usepackage{subfigure}
\usepackage[algo2e,linesnumbered,ruled]{algorithm2e}

\newtheorem{theorem}{Theorem}[section]
\newtheorem{corollary}[theorem]{Corollary}
\newtheorem{lemma}[theorem]{Lemma}

\newtheorem{assumption}[theorem]{Assumption}

\newtheorem{example}[theorem]{Example}
\newtheorem{remark}[theorem]{Remark}

\numberwithin{equation}{section}  

  \newcounter{mnote}
  \let\oldmarginpar\marginpar
    \renewcommand\marginpar[1]{\-\oldmarginpar[\raggedleft\footnotesize #1]%
    {\raggedright\footnotesize #1}}

\definecolor{myblue}{rgb}{0.2,0.2,0.7}
\definecolor{mygreen}{rgb}{0,0.6,0}
\definecolor{mycyan}{rgb}{0,0.6,0.6}
\definecolor{myred}{rgb}{0.9,0.2,0.2}
\definecolor{mymagenta}{rgb}{0.9,0.2,0.9}
\definecolor{mywhite}{rgb}{1.0,1.0,1.0}
\definecolor{myblack}{rgb}{0.0,0.0,0.0}

\newcommand{\beq}{\begin{equation}}
\newcommand{\eeq}{\end{equation}}
\newcommand{\beqa}{\begin{eqnarray}}
\newcommand{\eeqa}{\end{eqnarray}}
\newcommand{\tbar}{|\hspace*{-0.15em}|\hspace*{-0.15em}|}

\newcommand{\leqs}{\leqslant}      
     
\newcommand{\geqs}{\geqslant}      
      
\newcommand{\R}{{\mathbb R}}       

\newcommand{\cT}{{\mathcal T}}

\begin{document}

\title[Two-Grid Methods for Semilinear Interface Problems]
      {Two-Grid Methods for Semilinear Interface problems}

\author[M. Holst]{Michael Holst}
\email{mholst@math.ucsd.edu}

\author[R. Szypowski]{Ryan Szypowski}
\email{rsszypowski@csupomona.edu}

\author[Y. Zhu]{Yunrong Zhu}
\email{zhu@math.ucsd.edu}

\address{Department of Mathematics\\
         University of California San Diego\\ 
         La Jolla CA 92093}
\address{Department of Mathematics and Statistics\\
         California State Polytechnic University, Pomona\\ 
         Pomona CA 91768}

\thanks{MH was supported in part by NSF Awards~0715146 and 0915220,
and by DOD/DTRA Award HDTRA-09-1-0036.}
\thanks{RS and YZ were supported in part by NSF Award~0715146.}

\date{\today}


\begin{abstract}
In this article we consider two-grid finite element methods for solving 
semilinear interface problems in $d$ space dimensions, for $d=2$ or $d=3$.
We first describe in some detail the target problem class with 
discontinuous diffusion coefficients, which includes problems containing
sub-critical, critical, and supercritical nonlinearities.
We then establish basic quasi-optimal \emph{a priori} error estimate
for Galerkin approximations.
In the critical and subcritical cases, we follow our recent approach
to controling the nonlinearity using only pointwise control of the 
continuous solution and a local Lipschitz property, rather than 
through pointwise control of the discrete solution;
this eliminates the requirement that the discrete solution satisfy a 
discrete form of the maximum principle, hence eliminating the need for 
restrictive angle conditions in the underlying mesh.
The supercritical case continues to require such mesh conditions in order to 
control the nonlinearity.
We then design a two-grid algorithm consisting of a coarse grid solver 
for the original nonlinear problem, and a fine grid solver for a 
linearized problem. 
We analyze the quality of approximations generated by the algorithm,
and show that the coarse grid may be taken to have much larger elements
than the fine grid, and yet one can still obtain approximation quality that 
is asymptotically as good as solving the original nonlinear problem on the
fine mesh.
The algorithm we describe, and its analysis in this article, 
combines four sets of tools: 
the work of Xu and Zhou on two-grid algorithms for 
semilinear problems;
the recent results for linear interface problems due to 
Li, Melenk, Wohlmuth, and Zou;
and recent work on \emph{a priori} estimates for semilinear problems.
\end{abstract}

\maketitle


\tableofcontents

\clearpage

\section{Introduction}
   \label{sec:intro}

In this article, we consider a two-grid finite element method for
semilinear interface problems with discontinuous diffusion
coefficients. 
One of the primary motivations of this work is to develop more efficient
numerical methods for the nonlinear Poisson-Boltzmann equation, 
which has important applications in biochemistry and biophysics. 
However, the theory and techniques are applicable to a large class of 
semilinear interface problems, including problems with critical 
(and subcritical) nonlinearity arising in geometric analysis 
and mathematical general relativity. 


In order to achieve our goal of exploiting two-grid-type discretizations,
our first task is to more completely develop a basic quasi-optimal 
\emph{a priori} error analysis for Galerkin approximations of semilinear
interface problems.
The main challenge comes from the loss of global regularity for interface 
problems (cf. \cite{Babuska.I1970,Kellogg.R1975}). 
There has been much work on finite element approximation of the linear 
elliptic interface problem. 
For example, in~\cite{Babuska.I1970} an equivalent minimization problem was 
introduced to handle the jump interface condition; this problem was then
solved using finite element methods.
Subsequently, finite element approximation of the elliptic interface problems
was analyzed using penalty methods in~\cite{Barrett.J;Elliott.C1987}, and 
optimal rates in the $H^{1}$ and $L^{2}$ norms were obtained by appropriately 
choosing the penalty parameter.
Optimal \emph{a priori} error estimates for linear interface problems in the 
energy norm (i.e., a weighted $H^{1}$ norm) is given 
in~\cite{Plum.M;Wieners.C2003}.
In~\cite{Xu.J1982}, suboptimal error estimates of 
order $O(h|\log h|^{1/2})$ in the $H^{1}$ norm was obtained for 2D linear 
interface problems using standard finite element techniques.
Similarly, in~\cite{Chen.Z;Zou.J1998} it was shown that for $C^{2}$ interfaces
in 2D convex polygonal domains $\Gamma$, the linear FEM approximation
$u_{h}$ has suboptimal standard error estimates of orders $O(h|\log h|^{1/2})$ 
and $O(h^{2}|\log h|^{1/2})$ in $H^{1}$ and $L^{2}$ norms respectively. 
By using isoparametric elements to fit the smooth interface, optimal error 
estimates were obtained in~\cite{Sinha.R;Deka.B2006} for 2D interface problems.
These results have been generalized to higher-order finite elements 
approximation in~\cite{Li.J;Melenk.J;Wohlmuth.B;Zou.J2010}.
There are also other approaches for dealing with linear elliptic interface 
problem; for example, immersed interface finite element methods based on 
Cartesian grids (cf. \cite{Li.Z;Lin.T;Wu.X2003}), mortar finite element
(cf. \cite{Lamichhane.B;Wohlmuth.B2004}), and Lagrange multiplier methods
using non-matching meshes 
(cf. \cite{Hansbo.P;Lovadina.C;Perugia.I;Sangalli.G2005}).

Less work has been done for nonlinear interface problems. 
For smooth coefficients under quite strong (global) regularity assumptions,
quasi-optimal error estimates were obtained 
by~\cite{Xu.J1996b,Xu.J;Zhou.A2001a}. 
Due to the loss of global regularity for interface problems 
(cf.~\cite{Babuska.I1970,Kellogg.R1975}, see 
also~\cite{Nicaise.S;Sandig.A1994,Nicaise.S;Sandig.A1994a} 
for regularity of linear interface problems), these analysis techniques
are not applicable here.
Recently, Sinha and Deka~\cite{Sinha.R;Deka.B2009} studied linear finite 
element approximation of semilinear elliptic interface problems in two 
dimensional convex polygonal domains.
Under assumptions that the mesh resolves the interface, and that
the nonlinear function $b(\xi)$ satisfies
$$ |b'(\xi)| \le C|\xi|, \qquad \mbox{ and } |b''(\xi)| \leqs C, \qquad \forall \xi\in \R,$$ 
they showed optimal error estimates in the $H^{1}$ norm using the framework 
of~\cite{Brezzi.F;Rappaz.J;Raviart.P1980}, together with the results 
from~\cite{Chen.Z;Zou.J1998}. 

In this paper, we use a more natural approach for semilinear interface 
problems which can be applied to a somewhat different but larger class 
of nonlinear problems than~\cite{Sinha.R;Deka.B2009}.
For ease of exposition, we assume that the triangulation resolves the 
interface, although this assumption may be weakened.
The first step is to derive both continuous and discrete \emph{a priori}
$L^{\infty}$ bounds 
for the continuous and discrete solutions in order to control the nonlinearity.
While continuous $L^{\infty}$ bounds are fairly standard under quite general 
assumptions on the nonlinearity (cf. Assumption~\ref{ass:barrier}), 
discrete \emph{a priori} $L^{\infty}$ bounds require additional mesh conditions
on the triangulation (cf. Assumption~\ref{ass:angle}).
Based on \emph{a priori} $L^{\infty}$ control of the continuous and discrete 
solutions, we derive optimal \emph{a priori} error estimates in both the
$H^{1}$ and $L^{2}$ norms, with the help of a \emph{Local Monotonicity} 
assumption on the nonlinearity. 
A similar approach was used in \cite{Chen.L;Holst.M;Xu.J2007,Holst.M;McCammon.J;Yu.Z;Zhou.Y2012} for 
the Poisson-Boltzmann equation. 
We note the mesh conditions play a key role in obtaining discrete 
maximum/minimum principles (cf. \cite{Kerkhoven.T;Jerome.J1990,Jungel.A;Unterreiter.A2005,Karatson.J;Korotov.S2006,Wang.J;Zhang.R2011}). 
However, when the nonlinearity satisfies subcritical or critical growth 
conditions, and has some type of monotonicity, we have been able to derive 
quasi-optimal \emph{a priori} error estimates directly, 
without using discrete maximum principles, and hence without 
the need for any mesh angle condition assumptions~\cite{BHSZ11a}. 

Finite element approximation of semilinear interface problems results in 
the need to solve system of nonlinear algebraic equations, and the number
of unknowns in these systems can be extraordinary large in the case of three
or more spatial dimensions.
The most robust and efficient approach for solving these types of nonlinear 
algebraic systems has been repeatedly shown to be some variation of
damped inexact Newton iteration, which consists of an inner-outer iteration: 
an inner loop involving repeated linear solves, together with any outer 
loop involving a damped/inexact correction step. 
See for example 
\cite{Bank.R;Rose.D1981,Bank.R;Rose.D1982,Rannacher.R1991},
and also~\cite{Antal.I;Karatson.J2008} for an application to
nonlinear interface problems.
The basic approach involves the solution of a linear system on the fine mesh
at each Newton step.
However, the two-grid algorithm proposed 
in~\cite{Axelsson.O;Layton.W1996,Xu.J1996b} takes another approach, 
which consists of a coarse grid solver for the original nonlinear problem,
and a fine grid solver only involving for a linearized problem,
which is effectively a one-step Newton update of the solution. 
The benefit of using this two-grid idea is that it significantly reduces the 
overall computation cost, since we only need to solve the nonlinear problem 
on a coarse grid, and we can solve the linear problem on the fine grid by 
using standard multigrid/multilevel methods for optimal complexity. 
The central question concerning the two-grid method is to how choose the 
coarse grid problem; in other words, how coarse can one make the nonlinear 
problem discretization, but still achieve nearly optimal approximation 
properties if solving the full nonlinear problem on the fine grid.
Based on \emph{a priori} $H^{1}$ and $L^{2}$ error estimates for 
semilinear interface problems, in this paper we show that the basic framework 
developed in \cite{Xu.J1996b,Xu.J;Zhou.A2000} allows us to establish, 
both theoretically and numerically, that one may choose a coarse grid with 
much larger mesh size than the fine grid in the case of semilinear
interface problems.

The main contributions of this paper are as follows:
\begin{enumerate}
	\item We give a complete finite element error analysis for semilinear 
elliptic interface problems, under weak assumptions on the nonlinearity; 
this includes establishing quasi-optimal \emph{a priori} energy, 
$L^{2}$ and $L^{4}$ error estimates of the finite element approximation. 
	\item We also provide a practical approach to efficiently solve the 
resulting nonlinear algebraic problem by two-grid algorithms, reducing the 
solution of the original nonlinear system of equations on the fine grid 
to the solution of a nonlinear problem on a coarse grid having much fewer 
degrees of freedom, together with the solution of a linear problem on 
the fine mesh. 
We note that the resulting linear interface problem can be efficiently 
solved by PCG algorithms using multilevel or domain decomposition 
preconditioners (cf. \cite{Xu.J;Zhu.Y2008,Zhu.Y2008}). 
\end{enumerate}

The remainder of the article is organized as follows. 
In Section \ref{sec:pde}, we introduce the basic notation and the model 
problem. 
We also establish continuous $L^{\infty}$ bounds for the solution under very
weak assumptions on the data and the nonlinearity.
In Section~\ref{sec:fem}, we establish quasi-optimal error estimates for
the finite element approximation, by first deriving discrete \emph{a priori}
$L^{\infty}$ bounds. 
In Section~\ref{sec:two-grid}, we describe the two-grid algorithm, and give 
an analysis of the approximation properties of the algorithm. 
In Section~\ref{sec:num}, we give some numerical experiments 
to support our theoretical conclusions. 

\section{Semilinear Interface Problems}
   \label{sec:pde}

Let $\Omega \subset \R^{d}$ be a Lipschitz domain with $d \geqs 2$,
with an internal interface $\Gamma$ dividing it into two open disjoint
subdomains $\Omega_{1}$ and $\Omega_{2}$,
so that $\Omega=\Omega_{1}\cup \Gamma\cup \Omega_{2}$. 
For ease of exposition, we assume $\Omega_{1}$ and $\Omega_{2}$ are two 
non-overlapping polyhedral/polygonal subdomains.
We then focus on the following semilinear elliptic equation:
\begin{equation}
\label{eqn:model}
	-\nabla\cdot (D\nabla u)  +b(x, u) = 0 \mbox{  in  } \Omega, \qquad u|_{\partial \Omega} =0,
\end{equation}
with the jump conditions on $\Gamma$:
\begin{equation}
\label{eqn:interface}
	[u]= 0, \mbox{  and  } \left[D\frac{\partial u}{\partial {\bf n}}\right] =0 \quad \mbox{ on } \Gamma,
\end{equation}
where $[u]:= u_{1}|_{\Gamma} - u_{2}|_{\Gamma}$ and $\left[D\frac{\partial u}{\partial {\bf n}}\right] := D_{1}\frac{\partial{u_{1}}}{\partial {\bf n}_{1}} + D_{2}\frac{\partial{u_{2}}}{\partial {\bf n}_{2}}$ with ${\bf n}_{i}$ representing the unit outer normal on $\Omega_{i}$. Here $u_{i} \;\; (i=1,2)$ stands for the restriction of $u$ on $\Omega_{i}$.   We assume that the coefficient $D=D(x): \Omega \to \R^{d\times d}$ is symmetric and piecewise constant on each subdomain, i.e., $D=D_{i}$ in $\Omega_{i}$, and $D \in L^{\infty}(\Omega)$ satisfies
\begin{equation}
\label{eqn:diffusion-assump}
	m |\xi|^{2} \leqs \xi^{T} D(x) \xi \leqs M |\xi|^{2}, \quad \forall \xi \in \R^{d}, \quad x \in \Omega,
\end{equation}
for constants $m, M>0.$

In working with the solution and approximation theory for~\eqref{eqn:model}-\eqref{eqn:interface},
we will employ standard notation for the function spaces, norms,
and other objects that will be relevant.
For example, given any subset $G\subset \R^{d},$ we denote as $L^{p}(G)$ 
the Lebesgue spaces for ${1\leqs p\leqs \infty}$, with 
norm $\|\cdot \|_{0, p, G}$. 
We denote the Sobolev norms as ${\|v\|_{s,p,G} =\|v\|_{W^{s,p}(G)}}$ 
for the spaces $W^{s,p}(G)$, with ${W^{s,2}(G)=H^s(G)}$ when $p=2$.
For any two functions ${v\in L^{p}(G)}$ and ${w\in L^{q}(G)}$ 
with ${p,q\geqs 1}$ 
and ${1/p + 1/q =1}$, we denote the pairing ${(v, w)_{G} := \int_{G}vw dx}$. 
For simplicity, when $G = \Omega$, we omit it in the norms/pairings. 
We will also denote as $H^{s}(\Omega_{1}\cup \Omega_{2})$ 
the space of functions $u$ such that 
$u|_{\Omega_{i}}\in H^{s}(\Omega_{i})$ for $i=1,2$ and $s>1$, 
endowed with the  norm 
$$
	\|u\|_{H^{s}(\Omega_{1}\cup \Omega_{2})} ^{2}:= \|u\|_{H^{s}(\Omega_{1})}^{2} + \|u\|_{H^{s}(\Omega_{2})}^{2}.
$$
We will use the notation $x_1\lesssim y_1$, and $x_2\gtrsim y_2$,
whenever there exist constants $C_1, C_2$ independent of the mesh size
$h$ and the coefficient $D$ or other parameters that $x_1$,
$x_2$, $y_1$ and $y_2$ may depend on, and such that $x_1 \leqs C_1 y_1$
and $x_2\geqs C_2 y_2$. 
We also denote $x\simeq y$ as $C_{1} x \leqs y\leqs C_{2} x$. 
Without confusion, we will also write $b(\xi) := b(x, \xi)$ and $b'(\xi):= \partial b(x, \xi)/\partial \xi$ for simplicity. 

\begin{remark}
\label{rk:general}
Note that more general interface conditions
$\left[D\frac{\partial u^l}{\partial n_{\Gamma}}\right]_{\Gamma}= g_{\Gamma}$ 
for some given function $g_{\Gamma} \in L^{\infty}(\Gamma)$ and 
non-homogeneous Dirichlet data $u|_{\partial \Omega} = g$ can be easily
treated using our results here, due to the observation that one may split 
the equation into two sub-problems.
The first sub-problem is a linear elliptic interface problem, 
and the second sub-problem is a nonlinear elliptic problem 
\eqref{eqn:model} with homogeneous Dirichlet boundary condition. 
More precisely, let $u=u^l + u^n,$ where $u^l\in H_g^1(\Omega)$ satisfies 
the linear elliptic interface problem:
\begin{equation}
\label{eqn:linear}
	\left\{\begin{array}{lll}
		-\nabla \cdot(D \nabla u^l) &=& 0 \mbox{ in } \Omega\\
		u^l|_{\partial \Omega} = g , \mbox{ and } {\left[D\frac{\partial u^l}{\partial n_{\Gamma}}\right]}_{\Gamma}&=& g_{\Gamma}; 
	\end{array}\right.
\end{equation}
while the nonlinear part $u^n$ is the solution to the (homogeneous)
semilinear equation
\begin{equation*}
		-\nabla \cdot (D \nabla u^n) + b(u^n + u^l) = 0 \mbox{ in } \Omega, 
\end{equation*}
with the interface condition \eqref{eqn:interface}.
On the other hand, the treatment for the linear interface problem \eqref{eqn:linear} is 
standard; cf.~\cite{Chen.Z;Zou.J1998,Li.J;Melenk.J;Wohlmuth.B;Zou.J2010}.
Therefore, without loss of generality we focus on \eqref{eqn:model} with 
homogeneous interface conditions~\eqref{eqn:interface}.
\end{remark}

The weak form of~\eqref{eqn:model} reads: Find $u\in H_{0}^{1}(\Omega)$ such that 
\begin{equation}
\label{eqn:weak}
	a(u, v) + (b(u), v) =0, \qquad \forall v\in H_{0}^{1}(\Omega),
\end{equation}
where $a(u, v) := \int_{\Omega} D\nabla u \cdot\nabla v dx$.
By the assumption \eqref{eqn:diffusion-assump} on the coefficient $D$, 
the bilinear form $a(u, v)$ 
in~\eqref{eqn:weak} is coercive and continuous, namely, 
\begin{equation}
m \|\nabla u\|_{0,2}^2 \leqs a(u,u),
\qquad
a(u,v) \leqs M \|\nabla u\|_{0,2} \|\nabla v\|_{0,2}, 
\qquad \forall u,v \in H_0^{1}(\Omega),
\label{eqn:coercive-bounded}
\end{equation}
where $0 < m \leqs M < \infty$ are constants depending only on the
maximal and minimal eigenvalues on $D$ and the domain $\Omega$. 
The properties~\eqref{eqn:coercive-bounded}
imply the semi-norm on $H_0^{1}(\Omega)$ is equivalent
to the energy norm 
$\tbar\cdot\tbar \colon H_0^{1}(\Omega) \rightarrow \mathbb{R}$,
\begin{equation}
	\tbar u \tbar^2 = a(u,u),
	\qquad
	m \|\nabla u\|_{0,2}^2 \leqs \tbar u \tbar^2 \leqs M \|\nabla u\|_{0,2}^2.
   \label{eqn:equiv}
\end{equation}

\emph{A priori} $L^{\infty}$ bounds for any solution to the continuous problem
play a crucial role in controlling the nonlinearity. 
The following weak assumption on the nonlinearity allows for a large class 
of nonlinear problems containing both monotone and non-monotone nonlinearity:
\begin{assumption}
\label{ass:barrier}
$b:\Omega\times\R\to \R$ is a Carath\'eodory function, which satisfies the barrier-sign conditions in its second argument: there exist constants 
$\alpha, \beta \in \mathbb{R}$, with ${\alpha \leqs \beta}$, such that
\begin{align*}
b(x,\xi)&\geqs 0, \quad \forall \xi \geqs \beta, 
\quad \mbox{ a.e. in } \Omega
\\
b(x,\xi)&\leqs 0, \quad \forall \xi \leqs \alpha, 
\quad \mbox{ a.e. in } \Omega.
\end{align*}
\end{assumption}
We have the following theorem based on the Assumptions~\ref{ass:barrier}:
\begin{theorem}[\emph{A Priori} $L^{\infty}$ Bounds]
\label{thm:cont-infty}
Let the Assumption~\ref{ass:barrier} hold.
Let $u\in H_g^1(\Omega)$ be any
weak solution to~\eqref{eqn:weak}.
Then  
\begin{equation}
\label{eqn:gcont-infty}
\underline{u} \leqs u \leqs \overline{u}, \quad \mbox{ a.e. in } \Omega,
\end{equation}
for the constants $\overline{u}$ and $\underline{u}$ defined by
\begin{equation}
\label{eqn:barriers}
	\overline{u} := \max\left\{\beta, \sup_{x\in \partial \Omega} g(x)\right\},\qquad \underline{u}: = \min\left\{\alpha, \inf_{x\in \partial \Omega} g(x)\right\},
\end{equation} 
where $\alpha \leqs \beta$ are the constants in Assumption~\ref{ass:barrier}. 
\end{theorem}
\begin{proof}
To prove the upper bound, let us introduce
$$
\phi = (u-\overline{u})^+=\max\{ u-\overline{u}, 0\}.
$$
By the definition of $\overline{u}$, 
it follows (cf.~\cite[Theorem 10.3.8]{StHo2011a})
that $\phi \in H_{0}^{1}(\Omega)$ and $\phi \geqs 0$ a.e. in $\Omega$. 
Taking $v= \phi$ in \eqref{eqn:weak}, we have 
\begin{eqnarray*}
	a(u, \phi) = a(u-\overline{u}, \phi) = a(\phi, \phi) \geqs m\|\nabla \phi\|_{0,2}^{2}.
\end{eqnarray*}
This implies that 
$$
	m\|\nabla \phi\|_{0,2}^{2} \leqs a(u, \phi) = (-b(u), \phi) \leqs 0,
$$
since $-b(u) \leqs 0 $ a.e. in the support of $\phi$. Hence, $\|\nabla \phi\|_{0,2}\equiv 0$ which yields
$\phi =0$. Therefore, the upper bound of \eqref{eqn:gcont-infty} holds.

Similarly, we introduce 
$$
	\psi = (u-\underline{u})^-=\min\{ u-\underline{u}, 0\}.
$$
It is obvious that $\psi \in H_{0}^{1}(\Omega)$ can be used as a test function in \eqref{eqn:weak}.
Moreover, $\psi \leqs 0$ {a.e.} in $\Omega$, and Assumption~\ref{ass:barrier} implies
$- b(u) \geqs 0$ on the support of $\psi$. Therefore,
$$
	m\|\nabla \psi\|_{0,2}^{2} \leqs a(u, \psi) =  (-b(u), \psi) \leqs 0,
$$
which implies $\psi \equiv 0$ as before. This proves the lower bound of \eqref{eqn:gcont-infty}.
\end{proof}

To conclude this section, we give the nonlinear Poisson-Boltzmann equation as an example, which is one of our main motivation for this work. This equation has been widely used in biochemistry, biophysics and in semiconductor modeling for describing
the electrostatic interactions of charged bodies in dielectric media. 
\begin{example}
\label{ex:pbe}
The regularized Poisson-Boltzmann equation reads:
\begin{equation}
	\label{eqn:rpbe}
	\left\{\begin{array}{rlll}
	-\nabla\cdot (\varepsilon \nabla u) + \kappa^2 \sinh (u) &=& 0, &\mbox{in } \Omega\\
	
	[u]_{\Gamma} =0 \mbox{ and }{\left[\varepsilon \frac{\partial u}{\partial {\bf n}}\right]}_{\Gamma} &=& g_{\Gamma}, &\mbox{on } \Gamma\\

	u|_{\partial \Omega} & = & g, &\mbox{on } \partial\Omega,\\
	\end{array}\right.
\end{equation}
where $g_{\Gamma} \in L^{\infty}(\Gamma)$ is a function defined on $\Gamma$ arising from regularization of pointwise charges in the molecular region (see~\cite{Chen.L;Holst.M;Xu.J2007,Holst.M;McCammon.J;Yu.Z;Zhou.Y2012} for detailed derivations). 
Here the diffusion coefficient $\varepsilon$ is piecewise positive constant $\varepsilon|_{\Omega_{1}} = \varepsilon_{1}$ and $\varepsilon|_{\Omega_{2}} = \varepsilon_{2}$, where $\Omega_{1}$ is the molecular region, and $\Omega_{2}$ is the solution region(see Figure~\ref{fig:interface} for example). 
\begin{figure}[h]
\label{fig:interface}
	\includegraphics{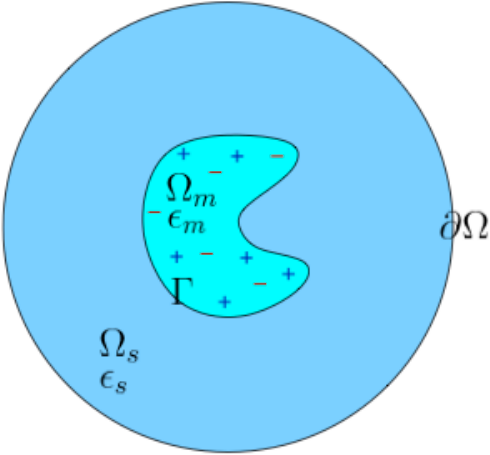}
\end{figure}
The modified Debye-H\"uckel parameter $\kappa^{2}$ is also piecewise constant $\kappa^{2}(x)|_{\Omega_{1}} = 0$ and $\kappa^{2}(x)|_{\Omega_{2}} >0$. The Dirichlet condition $u|_{\partial \Omega} = g$ are imposed on the boundary $\partial \Omega$.
We note that
equation \eqref{eqn:rpbe} can be reduced to \eqref{eqn:model} by splitting it into linear and nonlinear components as described in Remark~\ref{rk:general}, see \cite{Holst.M;McCammon.J;Yu.Z;Zhou.Y2012} for more details. 
Obviously, the Assumption~(A1) is satisfied for \eqref{eqn:rpbe}.
\end{example}
 
\section{Finite Element Error Estimates}
   \label{sec:fem}

We now discuss some error estimates on the finite element discretization 
of~\eqref{eqn:weak} which will play a key role in the two-grid analysis.
Given a quasi-uniform triangulation $\cT_{h}$ of $\Omega,$ we denote by
$V_{h}(g) \subset H^{1}_{g}(\Omega)$ the standard piecewise linear finite 
element space satisfying the Dirichlet boundary condition. 
For simplicity, we denote $V_{h}:=V_{h}(0)$. 
For ease of exposition, we assume the triangulation $\cT_{h}$ resolves the 
interface $\Gamma$. 
Then finite element approximation of the target problem~\eqref{eqn:model} 
reads: Find $u_{h} \in V_{h}$ such that
\begin{equation}
\label{eqn:fem}        
	a(u_{h}, v) + (b(u_{h}), v) = 0,\qquad \forall v\in V_{h}.
\end{equation}
The following theorem shows that under appropriate mesh condition on $\cT_{h}$,
the discrete solution $u_{h}$ of \eqref{eqn:fem} satisfies 
\emph{a priori} $L^{\infty}$ bounds (as does the continuous solution 
$u$ due to Theorem~\ref{thm:cont-infty}).
More precisely, assume the triangulation $\cT_{h}$ satisfies
\begin{assumption}
\label{ass:angle}
Let $\phi_{i}$ and $\phi_{j}$ are the basis functions corresponding to the vertices $x_{i}$ and $x_{j}$, respectively. We assume that
\begin{equation}
	a(\phi_{i}, \phi_{j}) \leqs 0, \qquad \forall i\neq j,
\end{equation}
\end{assumption}
Under this assumption, we can obtain the following \emph{a priori} $L^{\infty}$ bound of the discrete solution $u_{h}$.
\begin{theorem}
\label{thm:disc-infty}
	Let $b$ satisfy the Assumption ~\ref{ass:barrier}, and $\cT_{h}$ satisfy the Assumption ~\ref{ass:angle}.  Then the discrete solution $u_{h}\in V_{h}(g)$ to \eqref{eqn:fem} satisfies 
	\begin{equation}
		\label{eqn:disc-infty}
		\underline{u} \leqs u_{h} \leqs \overline{u}, 	\quad a.e.\;\; \mbox{ in } \Omega,
	\end{equation}
	where $\underline{u}$ and $\overline{u}$ are defined in \eqref{eqn:barriers}.
\end{theorem}
The idea of the proof of~\eqref{eqn:disc-infty} is the same as in 
Theorem~\ref{thm:cont-infty}. 
However, in the discrete setting, for a given $r\in \R$ the truncated 
functions $(u_{h} -r)^{\pm}$ are usually not in $V_{h}$. 
Thus, they can not be used as test functions in \eqref{eqn:fem}.
Instead, one can employ the nodal interpolation of these functions as test 
functions.
In particular, given any constant $r$, we denote
$$
	[u_{h} -r]^{\pm}:= \sum_{i=1}^{N} (u_{h}(x_{i}) - r)^{\pm} \phi_{i},
$$
where $N$ is the total number of degree of freedoms, and $\phi_{i}$  ($i=1, \cdots, N$) is the nodal basis function at the vertex $x_{i}$. 
While this does produce proper test functions, it unfortunately introduces 
mesh conditions such as Assumption~\ref{ass:angle} into the analysis.
\begin{proof}[Proof of Theorem~\ref{thm:disc-infty}]
	To prove the upper bound of \eqref{eqn:disc-infty}, define a test function $\phi_{h}^{+}(x) := [u_{h}(x)- \overline{u}]^{+} $. It is obvious that  $\phi_{h}^{+} \in V_{h}$ and ${\rm supp}(\phi_{h}^{+})$ is the union of of the macro elements for the vertices such that $u_{h}(x_{i}) >\overline{u}$. Therefore, it satisfies 
	\begin{equation}	
	\label{eqn:cutfem}
		a(u_{h}, \phi_{h}^{+}) + (b(u_{h}), \phi_{h}^{+}) =0.
	\end{equation}
	For the diffusion term in \eqref{eqn:cutfem}, we notice that
	\begin{eqnarray*}
		a(u_{h}, \phi_{h}^{+}) &= & a(\phi_{h}^{+}, \phi_{h}^{+}) + a([u_{h} -\overline{u}]^{-}, \phi_{h}^{+})\\
		&=&  a(\phi_{h}^{+}, \phi_{h}^{+}) + \sum_{i\neq j} (u_{h}(x_{i}) -\overline{u})^{-} (u_{h}(x_{j}) -\overline{u})^{+}a(\phi_{i}, \phi_{j})\\
		&\geqs& a(\phi_{h}^{+}, \phi_{h}^{+}) \geqs m\|\nabla \phi_{h}^{+} \|_{0,2}^{2},
	\end{eqnarray*}
	where we used the Assumption ~\ref{ass:angle} in the third step, and used \eqref{eqn:diffusion-assump} in the last step.
	Thus we obtain that
	$$
		m\|\nabla \phi_{h}^{+} \|_{0,2}^{2} \leqs a(u_{h}, \phi_{h}^{+}) = (-b(u_{h}), \phi_{h}^{+}) \leqs 0.
	$$
	This implies that $\phi_{h}^{+} \equiv 0$, and hence $u_{h}\leqs \overline{u}$ a.e. in $\Omega$.  The proof of the lower bound is similar, and so
we omit the detail here.
\end{proof}

\begin{remark}
\label{rk:noangle}
	Note that Assumption~\ref{ass:angle} requires certain angle condition on the triangulation $\cT_{h}$. This condition is crucial in proving the discrete maximal/minimal principle (cf. \cite{Ciarlet.P;Raviart.P1973,Kerkhoven.T;Jerome.J1990,Jungel.A;Unterreiter.A2005,Wang.J;Zhang.R2011}). However, in case that $b$ satisfies critical/subcritical growth condition, namely, there exists some constant $K>0$ such that 
	\begin{equation}
	\label{eqn:critical}
		|b^{(n)}(\xi)|\le K, \quad \forall \xi\in \R
	\end{equation} 
	where $n$ is an integer satisfying $n<\infty$ when $d=2$ and $n\leqs (d+2)/(d-2)$ when $d\geqs 3,$ 
	we are able to show the quasi-optimal error estimate directly, without using Assumption~\ref{ass:angle}; see \cite{BHSZ11a} for more detail.
\end{remark}

\emph{A priori} $L^{\infty}$ bounds (Theorem~\ref{thm:cont-infty} and Theorem~\ref{thm:disc-infty}) play crucial roles in controlling the nonlinearity, 
ensuring that the nonlinearity $b$ has a certain ``local Lipschitz'' property.
This property in turn is used to establish quasi-optimal error estimates for 
the finite element approximations.
For this purpose, let us make the following additional assumption on $b:$
\begin{assumption}
\label{ass:bm}
	$b$ is locally monotone, namely, 
	\begin{equation}
	\label{eqn:bm}
		b'(\xi) \geqs 0, \qquad \forall \xi \in [\underline{u}, \overline{u}],
	\end{equation}
	where $\underline{u}, \overline{u}$ are the barriers defined in \eqref{eqn:barriers}.
\end{assumption}
Without loss of generality, in the remainder of the paper, we let the Dirichlet data ${g=0}$. With the help of the \emph{a priori} $L^{\infty}$ bounds of $u$ and $u_{h}$ in Theorem~\ref{thm:cont-infty} and Theorem~\ref{thm:disc-infty} respectively, we are able to establish the following quasi-optimal error estimate.
\begin{theorem}
	\label{thm:quasi_optimal}
	Let $b$ satisfies the Assumptions~\ref{ass:barrier} and~\ref{ass:bm}, and $\cT_{h}$ satisfy the Assumption~\ref{ass:angle}.
Let $u\in H_{0}^{1}(\Omega)\cap H^{s}(\Omega_{1}\cup \Omega_{2})$ for some $s > 1$ and $u_{h}\in V_{h}$ be the solution to \eqref{eqn:weak} and the discrete solution to~\eqref{eqn:fem}, respectively.  
	Then the following quasi-optimal error estimate holds:
	\begin{equation}
	\label{eqn:error}
		\tbar u-u_h\tbar  \lesssim \inf_{v \in V_h} \tbar u-v\tbar \lesssim h^{s-1} \|u\|_{H^{s}(\Omega_{1}\cup\Omega_{2})}.
	\end{equation}
\end{theorem}
\begin{proof}
	By Assumptions~\ref{ass:barrier} and~\ref{ass:angle}, 
Theorem~\ref{thm:cont-infty} and ~\ref{thm:disc-infty} 
give \emph{a priori} $L^{\infty}$ bounds on $u$ and $u_{h}$: 
	$$\underline{u} \leqs u, u_{h} \leqs \overline{u}$$ 
	 for the constants $\underline{u}\leqs \overline{u}$ defined in \eqref{eqn:barriers}. This implies that
	\begin{equation}
		\label{eqn:lipschitz}
		b(u) - b(u_{h}) = b'(\xi) (u-u_{h}) \leqs C_L |u -u_{h}|,
	\end{equation}
	where $C_L = \sup_{\xi \in [\underline{u}, \overline{u}]} \|b'(\xi)\|_{0,\infty}$ is a constant depending only on $b$, $\underline{u}$ and $\overline{u}$.
	Subtracting equation \eqref{eqn:fem} from \eqref{eqn:weak}, we have 
\begin{equation*}
        a(u-u_{h}, v) + (b(u) - b(u_{h}), v) =0, \qquad \forall v\in V_{h}.
\end{equation*}
By using this identity, we obtain
\begin{align*}
\tbar u - u_{h}\tbar^{2}
  &= a(u-u_{h}, u-u_{h}) \\
  &= a(u-u_{h}, u-v) + a(u-u_{h}, v-u_{h}), \quad \forall v\in V_{h}\\
  &= a(u-u_{h}, u-v) + (b(u) - b(u_{h}), u_{h} -v)\\
  &= a(u-u_{h}, u-v) + (b(u) - b(u_{h}), u -v) - (b'(\xi)(u -u_{h}), u - u_{h})\\
  &\leqs \tbar u-u_{h}\tbar  \tbar u - v\tbar 
  +C_L\|u -u_{h}\|_{0,2} \|u-v\|_{0,2},
\end{align*}
where in the last inequality, we used Cauchy-Schwarz inequality, the Lipschitz property \eqref{eqn:lipschitz} of $b$ and the Local Monotonicity \eqref{eqn:bm} from Assumption~\ref{ass:bm}. 
Then by 
Poincar\'e inequality we have
\begin{align*}
\tbar u - u_{h}\tbar^{2}
  &\leqs \tbar u- u_{h}\tbar \tbar u - v\tbar 
   + C_LC_{P}^{2} \|\nabla(u - u_{h})\|_{0,2} \|\nabla(u - v)\|_{0,2}
\end{align*}
where $C_{P}$ is the Poincar\'e constant. 
Thus we obtain
$$
	\tbar u - u_{h}\tbar \lesssim \tbar u - v\tbar.
$$
Therefore, we have proved the first inequality in~\eqref{eqn:error},
since $v\in V_{h}$ is arbitrary.
The second inequality in \eqref{eqn:error} follows by standard interpolation 
error estimates; cf.~\cite[Theorem 3.5]{Li.J;Melenk.J;Wohlmuth.B;Zou.J2010}.
\end{proof}

To conclude this section, let us try to derive $L^{2}$ error estimates:
$\|u-u_{h}\|_{0,2}$, by using duality arguments. 
To begin, introduce the following linear adjoint problem: 
Find $w\in H_{0}^{1}(\Omega)$ such that
\begin{equation}
	\label{eqn:dual}
	 a(v, w) + (b'(u) v, w) = (u - u_{h}, v), 	\qquad v \in H_{0}^{1}(\Omega).
\end{equation}
We assume that the linear interface problem~\eqref{eqn:dual} has the regularity 
\begin{equation}
\label{eqn:reg}
	\|w\|_{H^{\tau}(\Omega_{1} \cup \Omega_{2})} \leqs C \|u - u_{h}\|_{0,2}
\end{equation}
for some $\tau >1.$ 
The regularity assumption \eqref{eqn:reg} is also called ``$\tau$-regularity'' in \cite[Assumption 4.3]{Li.J;Melenk.J;Wohlmuth.B;Zou.J2010}, which is quite natural for linear interface problems.
Along with \eqref{eqn:dual}, let us also introduce the finite element 
approximation $w_{h}\in V_{h}$ which satisfies:
\begin{equation}
	\label{eqn:dualfem}
	 a(v_{h}, w_{h}) + (b'(u) v_{h}, w_{h}) = (u - u_{h}, v_{h}), 	\qquad v_{h} \in V_{h}.
\end{equation}
Then by standard finite element approximation theory for the linear interface 
problem~\eqref{eqn:dual} 
(cf. \cite{Chen.Z;Zou.J1998,Li.J;Melenk.J;Wohlmuth.B;Zou.J2010}),
we have the following error estimate:
\begin{equation}
\label{eqn:femlinear}
	\tbar w -w_{h}\tbar \lesssim h^{\tau-1} \|w\|_{H^{\tau}(\Omega_{1}\cup \Omega_{2})} \lesssim h^{\tau-1} \|u-u_{h}\|_{0,2},
\end{equation}
where in the second inequality we have used the regularity 
assumption~\eqref{eqn:reg}. 
We then have the following $L^{2}$ error estimate for $u_{h}$:
\begin{theorem}[$L^{2}$ Error Estimate]
	\label{thm:l2err}
	Let $b$ satisfy Assumptions~\ref{ass:barrier} and \ref{ass:bm}, and 
let $\cT_{h}$ satisfy Assumption~\ref{ass:angle}. 
Let $u \in H_{0}^{1}(\Omega)\cap H^{s}(\Omega_{1} \cup \Omega_{2})$ with 
$s >1$ be the solution to~\eqref{eqn:weak}, and let $u_{h}$ be the solution 
to~\eqref{eqn:fem}. 
Suppose that the dual problem \eqref{eqn:dual} satisfies the 
$\tau$-regularity \eqref{eqn:reg} for some $\tau >1$. 
Then
	\begin{equation}
	\label{eqn:l2}
		\|u - u_{h}\|_{0,2} \leqs C (h^{\tau -1}  + h^{s-1}) \|\nabla (u - u_{h})\|_{0,2},
	\end{equation}
	where $C$ is independent of $h.$
\end{theorem}
\begin{proof}
Without loss of generality we assume $b''(\chi)\neq 0.$
By taking $v = u -u_{h}$ in~\eqref{eqn:dual} we obtain that
\begin{align}
	\|u - u_{h}\|^{2}_{0,2} &= a(u-u_{h}, w) + (b'(u) (u-u_{h}), w)\nonumber\\
	&= a(u-u_{h}, w- w_{h}) + (b'(u)(u-u_{h}), w-w_{h})\nonumber \\
	& \qquad + a(u-u_{h}, w_{h}) + (b'(u) (u-u_{h}), w_{h}). \label{eqn:l2eq}
\end{align}
To bound the second term in \eqref{eqn:l2eq}, we use the $L^{\infty}$ bound of $u$ (cf. Theorem~\ref{thm:cont-infty}) to obtain
\begin{eqnarray*}
	 (b'(u)(u-u_{h}), w-w_{h}) &\leqs& \|b'(u)\|_{0,\infty} \|u -u_{h}\|_{0,2} \|w-w_{h}\|_{0,2}\\
	  &\lesssim& \|\nabla (u-u_{h})\|_{0,2}\|\nabla (w-w_{h})\|_{0,2},
\end{eqnarray*}
where we have used Poincar\'e inequality for $(u-u_{h})$ and $(w-w_{h})$ in the last step.
To deal with the last two terms in \eqref{eqn:l2eq},  notice that $u_{h}\in V_{h}$ is the solution to the discrete semilinear problem~\eqref{eqn:fem}, we have
$$		 
	a(u-u_{h}, w_{h})  = - (b(u)-b(u_{h}), w_{h}).
$$
Thus by Taylor expansion, we have
\begin{eqnarray*}
	a(u-u_{h}, w_{h}) + (b'(u) (u-u_{h}), w_{h}) &=& -(b(u) -b(u_{h}) -(b'(u) (u-u_{h}), w_{h} )\\
	& =& \frac{1}{2} (b''(\chi) (u-u_{h})^{2}, w_{h}),
\end{eqnarray*}
where $\chi$ satisfies that $\underline{u} \leqs \chi(x) \leqs \overline{u}$ a.e. in $\Omega$ due to the \emph{a priori} $L^{\infty}$ bounds of $u$ and $u_{h}$ by Theorem~\ref{thm:cont-infty} and Theorem~\ref{thm:disc-infty}.
Therefore, by H\"older inequality the last two terms in \eqref{eqn:l2eq} can be bounded as: 
\begin{eqnarray*}
	a(u-u_{h}, w_{h}) + (b'(u) (u-u_{h}), w_{h}) &\leqs& \frac{1}{2} \|b''(\chi)\|_{0,\infty} \|u-u_{h}\|_{0,{p^{*}}} ^{2}\|w_{h}\|_{0,{q*}}\\
	&\lesssim& \| \nabla(u - u_{h})\|_{0,2}^{2} \|w_{h}\|_{0,{q^{*}}},
\end{eqnarray*}
where we choose $p^{*} = 6$ for $d=3$ and $p^{*}>4$ when $d=2,$ and $\frac{2}{p^{*}} + \frac{1}{q^{*}} =1.$ In the last inequality, we have used the Sobolev embedding $\|u-u_{h}\|_{0,p^{*}} \lesssim\|\nabla(u-u_{h})\|_{0,2}$. Therefore, we obtain
\begin{equation}
	\label{eqn:dual3}
	\|u - u_{h}\|^{2}_{0,2}  \lesssim ( \| \nabla(w-w_{h})\|_{0,2} +  \|\nabla( u-u_{h})\|_{0,2} \|w_{h}\|_{0,q^{*}} )\|\nabla(u -u_{h})\|_{0,2}.
\end{equation}

	Now in \eqref{eqn:dual3}, let $w_{h}\in V_{h}$ be the solution to \eqref{eqn:dualfem}. 
Then the estimate for the quantity
${\|\nabla(w - w_{h})\|_{0,2}}$ is readily available 
from \eqref{eqn:femlinear}. 
It then remains to estimate the term 
${\|\nabla (u-u_{h})\|_{0,2} \|w_{h}\|_{0,q^{*}}}$. 
To estimate $\|w_{h}\|_{0,q^{*}}$, notice that by the choice of 
${p^{*}}$, ${1<q^{*}<2}$. 
Then by Poincar\'e inequality and coercivity of $a(\cdot, \cdot)$
$$
	\|w_{h}\|_{0,q^{*}} \lesssim  \|w_{h}\|_{0,2} \lesssim \|\nabla w_{h}\|_{0,2} \lesssim \tbar w_{h} \tbar.
$$
By Assumption~\ref{ass:bm} on $b$, we have
\begin{align*}
	\tbar w_{h} \tbar^{2} &= a(w_{h}, w_{h})
	 \leqs a(w_{h}, w_{h}) + (b'(u) w_{h}, w_{h}) \\
	&= (u -u _{h}, w_{h}) \leqs \|u-u_{h}\|_{0,2} \|w_{h}\|_{0,2}\\
	&\lesssim \|u-u_{h}\|_{0,2} \tbar w_{h} \tbar.
\end{align*}
Thus, this inequality implies that
$$
	\|w_{h}\|_{0,q^{*}} \lesssim \tbar w_{h} \tbar \lesssim \|u-u_{h}\|_{0,2}.
$$
Therefore, we obtain
\begin{equation}
\label{eqn:dual4}
	\| \nabla(u-u_{h})\|_{0,2} \|w_{h}\|_{0,q^{*}} \lesssim h^{s-1} \|u\|_{H^{s}(\Omega_{1}\cup \Omega_{2})} \|u-u_{h}\|_{0,2}.
\end{equation}
Combining inequalities \eqref{eqn:dual3}, \eqref{eqn:femlinear} and \eqref{eqn:dual4}, the inequality \eqref{eqn:l2} then follows. This concludes the proof. 
\end{proof}
\begin{remark}
\label{rk:l2error}
	Theorem ~\ref{thm:l2err} implies, in particular, that 
	\begin{equation}
	\label{eqn:l2lifting}
		\|u - u_{h}\|_{0,2} \lesssim h^{t} \|\nabla(u -u_{h})\|_{0,2},
	\end{equation}
	where $t = \min\{s,\tau\} -1$. Combining with the quasi-optimal error estimate \eqref{eqn:error}, we obtain the following $L^{2}$ error estimate: 
	\begin{equation}
	\label{eqn:l2err}
		\|u - u_{h}\|_{0,2} \lesssim h^{t + (s-1)} \|u\|_{H^{s}(\Omega_{1}\cup\Omega_{2})}.
	\end{equation}
\end{remark}
\begin{remark}
	Similar to Remark \ref{rk:noangle}, if $b(\cdot)$ satisfies the critical/subcritical growth condition \eqref{eqn:critical}, then the conclusion of Theorem~\ref{thm:l2err} holds without the Assumption~\ref{ass:angle}. We refer to \cite{BHSZ11a} for the details.
\end{remark}

\section{Two-grid algorithms}
   \label{sec:two-grid}

We now consider a two-grid algorithm (cf.~\cite{Xu.J1996b}) to solve the finite element discretization~\eqref{eqn:fem} numerically. Let $\cT_{H}$ be a quasi-uniform triangulation with mesh size $H$, and $\cT_{h}$ with mesh size $h<H$ is a uniform refinement of $\cT_{H}$. We assume the triangulations satisfy Assumption~\eqref{ass:angle}. The algorithm consists of an \emph{exact} coarse solver on a coarse grid $\cT_{H}$, and a Newton update on the fine grid $\cT_{h}$.
In what follows, we will denote $u_{h}, u_{H}$ as the \emph{exact} finite element solutions to~\eqref{eqn:fem} on the grids $\cT_{h}$ and $\cT_{H},$ respectively. For simplicity, let us denote 
$$
	\langle F(u), \chi \rangle := a(u, \chi) + (b(u), \chi),
$$
and its linearization 
$$
	\langle F'(u) v, \chi \rangle:= a(v, \chi) + (b'(u)v, \chi).
$$
The two grid algorithm considered in this paper is as follows:
\begin{algorithm2e}\label{alg:1newton}
\caption{$u^{h}=\textsf{TwoGrid}\left(\cT_{H}, \cT_{h}\right)$}
\linesnotnumbered
Find $u_{H} \in V_{H}$ such that $$\langle F(u_{H}), v_{H} \rangle =0, \quad \forall v_{H} \in V_{H};$$

Find $u^{h} \in V_{h}$ such that $$\langle F'(u_{H}) u^{h}, v_{h}\rangle = \langle F'(u_{H}) u_{H}, v_{h}\rangle -\langle F(u_{H}), v_{h}\rangle \quad \forall v_{h} \in V_{h};$$
\end{algorithm2e}
The Algorithm~\ref{alg:1newton} solves the original nonlinear problem on the coarse grid $\cT_{H},$ and then performs one Newton iteration on the fine grid.

Fix any $\chi_{h}\in V_h,$ let $\eta(t):= \langle F(u_{H} + t(u_{h} - u_{H})), \chi_{h} \rangle.$ Then by Taylor expansion, we have
\begin{eqnarray*}
	0 &=& \langle F(u_{h}), \chi_{h} \rangle  =  \eta(1) = \eta(0) + \eta'(0) + \int_{0}^{1} \eta''(t)(1-t) dt\\
	 &=& \langle F(u_{H}), \chi_{h} \rangle + \langle F'(u_{H}) (u_{h} - u_{H}), \chi_{h} \rangle +  \int_{0}^{1} \eta''(t)(1-t) dt.
\end{eqnarray*}
Notice that by direct calculation, the remainder term, denoted by $R(u_{H}, u_{h}, \chi_{h}),$ has the following form:
\begin{eqnarray*}
	R(u_{H}, u_{h}, \chi_{h}) &:= &	 \int_{0}^{1} \eta''(t)(1-t) dt\\
	& = & \int_{0}^{1} (b''(u_{H} + t(u_{h} - u_{H})) (u_{h} - u_{H})^{2}, \chi_{h}) dt.
\end{eqnarray*}
By Theorem~\ref{thm:disc-infty}, we have $u_{h}, u_{H} \in [\underline{u}, \overline{u}]$. Therefore, we have the following estimate on $R(u_{H}, u_{h}, \chi_{h}):$
\begin{equation}
\label{eqn:residual}
	|R(u_{H}, u_{h}, \chi_{h})| \le C\|u_{H} - u_{h}\|_{0, 2p}^{2} \|\chi_{h}\|_{0,q},
\end{equation}
for any $p,q\ge 1$ with $\frac{1}{p} + \frac{1}{q} =1.$

\begin{lemma}
\label{lm:2grid}
	Let $b$ satisfy the Assumption~\ref{ass:barrier} and \ref{ass:bm}, and $\cT_{h}, \cT_{H}$ satisfy the Assumption~\ref{ass:angle}. Let $u_{H}\in V_{H}$ and $u_{h}\in V_{h}$ be the exact solutions to~\eqref{eqn:fem} on $\cT_{H}$ and $\cT_{h}$ respectively, and $u^{h}\in V_{h}$ be the approximated solution obtained by Algorithm~\ref{alg:1newton}. Then, we have the following estimate
\begin{equation}
	\tbar u_{h} - u^{h} \tbar \lesssim \|u_{h} - u_{H}\|_{0,4}^{2}.
\end{equation}
\end{lemma}
\begin{proof}
From Algorithm~\ref{alg:1newton}, we have
\begin{eqnarray*}
	\langle F'(u_{H}) (u_{h} - u^{h}), \chi_{h}\rangle & =& \langle F'(u_{H}) (u_{h} - u_{H}), \chi_{h} \rangle + \langle F(u_{H}), \chi_{h} \rangle \\
	&=& -R(u_{H}, u_{h}, \chi_{h}), \quad \forall \chi_{h} \in V_{h}.
\end{eqnarray*}
Then by taking $\chi_{h} = u_{h} - u^{h}\in V_{h}$ in the above equality, we obtain that
\begin{eqnarray*}
	\tbar u_{h} - u^{h} \tbar ^{2} &=& a(u_{h} - u^{h}, u_{h} - u^{h})\\
	& =& \langle F'(u_{H}) \chi_{h} , \chi_{h}\rangle - (b'(u_{H})\chi_{h}, \chi_{h})\\
	&\leqs & -R(u_{H}, u_{h}, u_{h} - u^{h}) \\
	&\lesssim & \|u_{h} - u_{H}\|_{0,2p}^{2} \|u_{h} - u^{h}\|_{0,q},
\end{eqnarray*}
where in the third step we used the Assumption~\ref{ass:bm}, and in the last step we used \eqref{eqn:residual}.
In particular, if we pick $p=q =2,$ the conclusion then follows by Poincar\'e inequality on $\|u_{h} - u^{h}\|_{0,2}.$
\end{proof}

Lemma~\ref{lm:2grid} suggests that we will need the $L^{4}$ error estimates $\|u-u_{h}\|_{0,4}$. For this purpose, we make use of the following Ladyzhenskaya's inequalities:
\begin{lemma}[{\cite[Lemma 1-2]{Ladyzhenskaya.O1969}}]
\label{lm:lady}
	For any $v\in H_{0}^{1}(\Omega),$ it holds
	\begin{equation}
\label{eqn:ladyzhenskaya-2D}
	\|v\|_{0,4} \leqs \sqrt[4]{2} \|v\|_{0,2}^{\frac{1}{2}} \|\nabla v\|_{0,2}^{\frac{1}{2}},\qquad d =2;
\end{equation}
and 
\begin{equation}
\label{eqn:ladyzhenskaya-3D}
	\|v\|_{0,4} \leqs \sqrt{2} \|v\|_{0,2}^{\frac{1}{4}} \|\nabla v\|_{0,2}^{\frac{3}{4}}, \qquad d  =3.
\end{equation}
\end{lemma}

Recall that we assume the solution to the original nonlinear problem \eqref{eqn:model} satisfies $u\in H_{0}^{1}(\Omega) \cap H^{s}(\Omega_{1} \cup \Omega_{2})$, and the dual linear problem ~\eqref{eqn:dual} has the regularity $w\in H_{0}^{1}(\Omega) \cap H^{\tau}(\Omega_{1} \cup \Omega_{2})$ for some $s,\tau>1$. We let $t = \min\{s, \tau\} -1$ as defined in \eqref{eqn:l2lifting}.  As a corollary of Lemma~\ref{lm:lady}, we obtain the $L^{4}$ error estimate:
\begin{corollary}
\label{cor:l4}
	Let $b$ satisfy the Assumptions~\ref{ass:barrier} and \ref{ass:bm}, and $\cT_{h}$ satisfy the Assumption~\ref{ass:angle}. Let $u \in H_{0}^{1}(\Omega)\cap H^{s}(\Omega_{1} \cup \Omega_{2})$ with $s >1$ be the solution to~\eqref{eqn:weak},  and $u_{h}\in V_{h}$ be the solution to \eqref{eqn:fem}. Suppose that the dual problem \eqref{eqn:dual} satisfies the $\tau$-regularity \eqref{eqn:reg} for some $\tau >1$.  Then the following error estimates hold:
	\begin{equation}
\label{eqn:l4-2D}
	\|u -u_{h}\|_{0,4} \lesssim h^{\frac{t}{2} +(s-1)} \|u\|_{H^{s}(\Omega_{1}\cup \Omega_{2})},\qquad d =2;
\end{equation}
and 
\begin{equation}
\label{eqn:l4-3D}
	\|u-u_{h}\|_{0,4} \lesssim h^{\frac{t}{4} + (s-1)} \|u\|_{H^{s}(\Omega_{1}\cup \Omega_{2})}, \qquad d=3,
\end{equation}
where $t = \min\{s, \tau\} -1.$
\end{corollary}
\begin{proof}
	The proof is simply a combination of Lemma~\ref{lm:lady} and the quasi-optimal error estimate \eqref{eqn:error} in Theorem~\ref{thm:quasi_optimal} and the $L^{2}$ error estimate in \eqref{eqn:l2lifting}. When $d=2$, by \eqref{eqn:ladyzhenskaya-2D} we have
	$$
		\|u - u_{h}\|_{0,4} \leqs \sqrt[4]{2} \|u-u_{h}\|_{0,2}^{\frac{1}{2}} \|\nabla (u - u_{h})\|_{0,2}^{\frac{1}{2}} \lesssim h^{t/2} \|\nabla (u-u_{h})\|_{0,2}.
	$$
	Thus, when $d=2$ we obtain 
	$$
	\|u -u_{h}\|_{0,4} \lesssim h^{\frac{t}{2} + (s-1)} \|u\|_{H^{s}(\Omega_{1}\cup \Omega_{2})}
	$$
	Similarly, when $d=3$ by~\eqref{eqn:ladyzhenskaya-3D} and \eqref{eqn:l2lifting} we obtain
	$$
		\|u - u_{h}\|_{0,4} \le Ch^{\frac{t}{4}} \|\nabla  (u - u_{h})\|_{0,2}.
	$$
	The conclusion then follows from the conclusion of Theorem~\ref{thm:quasi_optimal}.
\end{proof}
Finally, we obtain  the following main result.
\begin{theorem}
\label{thm:main}
	Let $b$ satisfy the Assumptions~\ref{ass:barrier} and \ref{ass:bm}, and $\cT_{h}, \cT_{H}$ satisfy the Assumption~\ref{ass:angle}. Let $u \in H_{0}^{1}(\Omega)\cap H^{s}(\Omega_{1} \cup \Omega_{2})$ with $s >1$ be the solution to~\eqref{eqn:weak},  and $u^{h}$ be the solution to Algorithm~\ref{alg:1newton}. Suppose that the dual problem \eqref{eqn:dual} satisfies the $\tau$-regularity \eqref{eqn:reg} for some $\tau >1$.  We have the following estimates
	\begin{equation}
		\tbar u - u^{h}\tbar \lesssim (h^{s-1} +  H^{t + 2(s -1)}) \|u\|_{H^{s}(\Omega_{1}\cup\Omega_{2})}, \qquad d=2,
	\end{equation}
	and
	\begin{equation}
		\tbar u - u^{h}\tbar \lesssim (h^{s-1} + H^{\frac{t}{2} + 2(s -1)}) \|u\|_{H^{s}(\Omega_{1}\cup\Omega_{2})}, \qquad d=3,
	\end{equation}
	where $t = \min\{s, \tau\} -1 >0.$
\end{theorem}
\begin{proof}
By triangle inequality, we have
\begin{eqnarray}
	\tbar u - u^{h} \tbar &\le& \tbar u - u_{h} \tbar + \tbar u_{h} - u^{h} \tbar \nonumber\\
	& \lesssim& \tbar u - u_{h} \tbar + \|u_{h} - u_{H}\|_{0,4}^{2}\nonumber\\
	& \lesssim&   \tbar u - u_{h} \tbar + \|u - u_{h}\|_{0,4}^{2} + \|u-u_{H}\|_{0,4}^{2}. \label{eqn:u-ush}
\end{eqnarray}
The first term in the right hand side of~\eqref{eqn:u-ush} has been estimated in Theorem ~\ref{thm:quasi_optimal}. Thus the conclusions immediately follow by applying Corollary~\ref{cor:l4} to the last two terms in \eqref{eqn:u-ush}.
\end{proof}
\begin{remark}
	Based on Theorem~\ref{thm:main}, we  may choose $H\leqs h^{\frac{s-1}{t+2(s-1)}}$ in 2D and $H \le h^{\frac{s-1}{t/2 + 2(s-1)}}$ in 3D, but still achieve quasi-optimal error estimate. In particularly, if the linear dual problem \eqref{eqn:dual} has the same or more regularity than the primal nonlinear problem \eqref{eqn:model}, i.e., $\tau\geqs s>1$, then $t=s-1$. So, in 2D case we may choose $H\leqs h^{1/3}$ and $H\leqs h^{2/5}$ in 3D.
\end{remark}

\section{Numerical Experiments}
\label{sec:num}

In this section, we present some numerical experiments to justify the theories. 
%
%
Here we consider solving the following semilinear equation
$$-\nabla\cdot (D\nabla u) + u^{11} =1000\delta_{0},\quad  u|_{\partial \Omega} =0,$$
where the diffusion coefficient  $D=1000$ inside $[-1/2, 1/2]^2$ and 1 outside, and $\delta_{0}$ is the delta function on origin. See Figure~\ref{fig:sol2d} for the solution of this equation. 
\begin{figure}[htp]
	\begin{center}
		\includegraphics[width=0.65\linewidth]{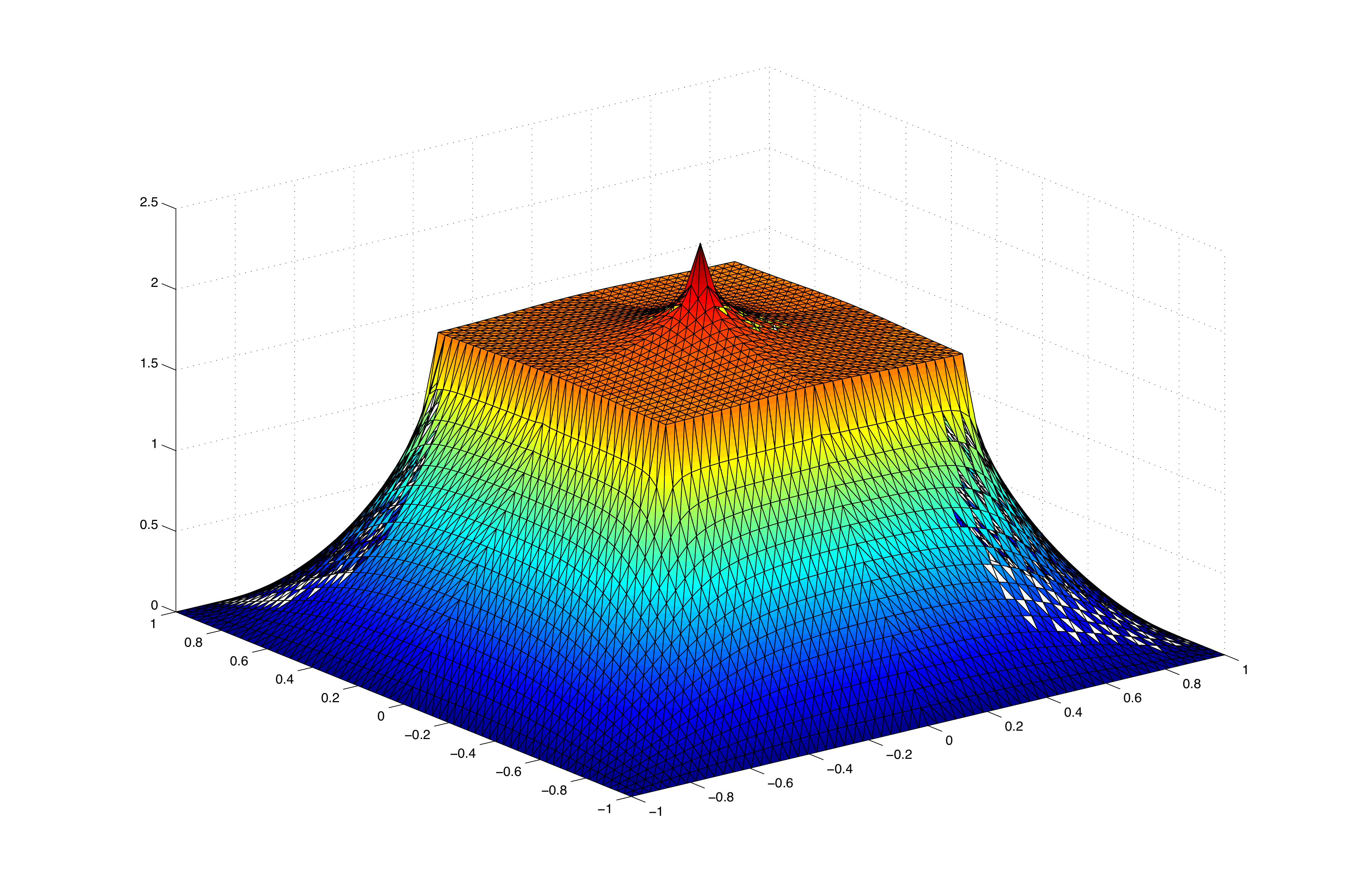}
	\end{center}
	\caption{Solution to the 2D semilinear interface problem.} \label{fig:sol2d}
\end{figure}
Figure~\ref{fig:h1} and~\ref{fig:l4} show the $H^{1}$ and $L^{4}$ errors, respectively.
 \begin{figure}[htp]
	\begin{center}
    	         \includegraphics[width=0.64\linewidth]{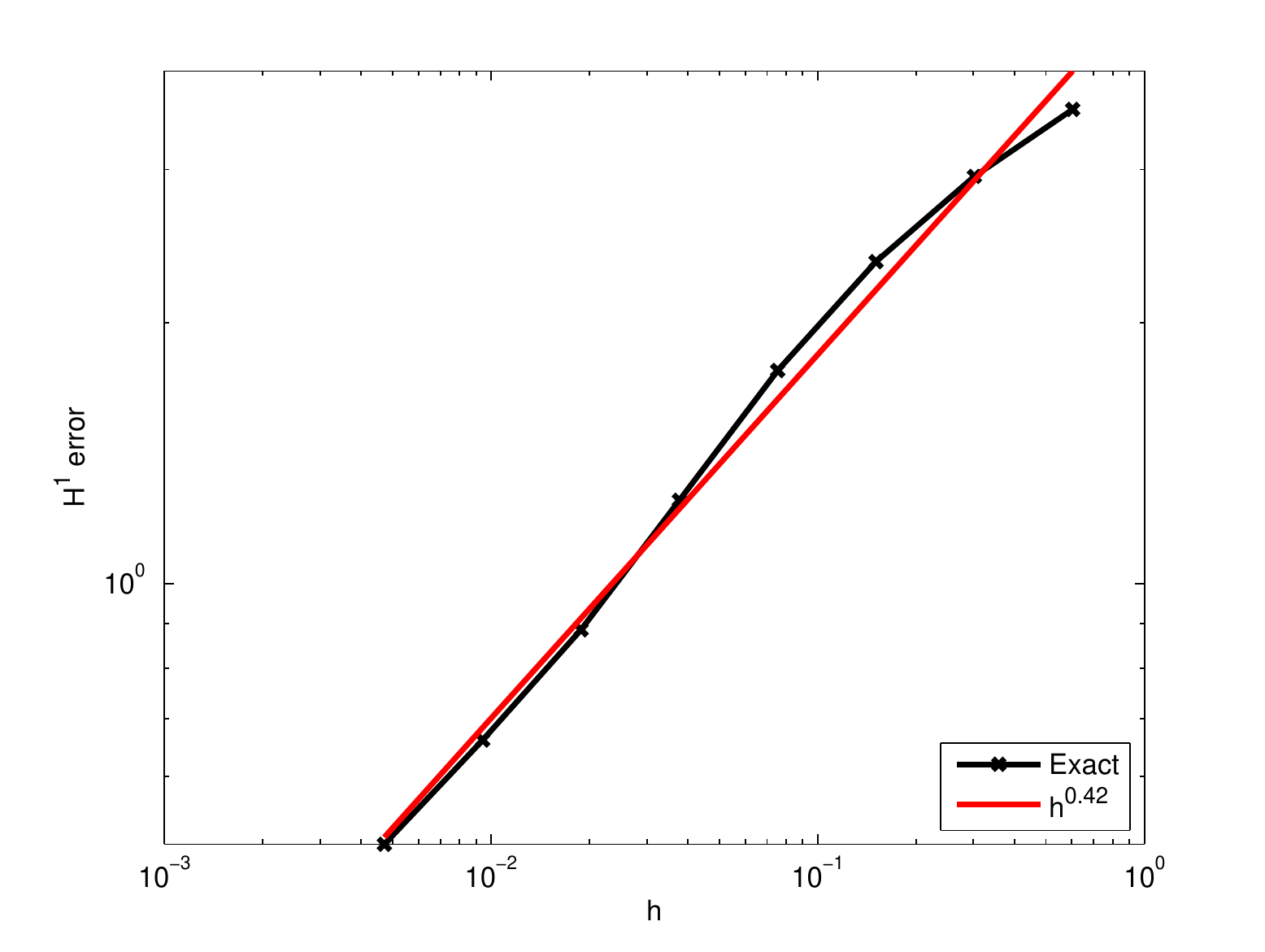}
	         \caption{Errors in $H^{1}$-norm.}\label{fig:h1}
	\end{center}
\end{figure}	
\begin{figure}[htp]
	\begin{center}
    	         \includegraphics[width=0.75\linewidth]{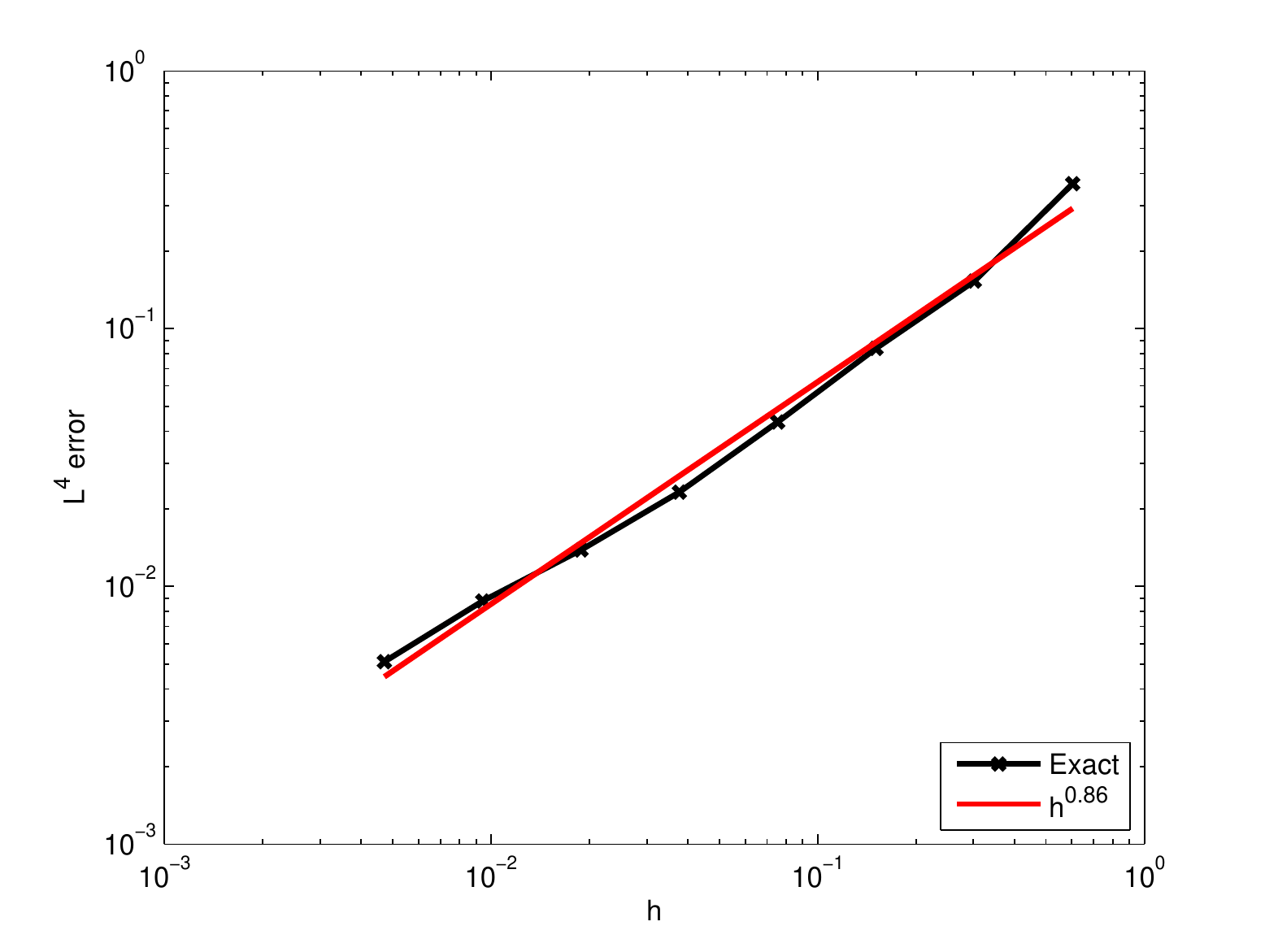}
	         \caption{Errors in $L^{4}$ norm.}\label{fig:l4}
	\end{center}
\end{figure}	
Figure~\ref{fig:2grid} shows the comparison of the exact $H^{1}$ error with the error of the two-grid solution produced by the Algorithm~\ref{alg:1newton}. Here, the mesh size $H$ of the coarse grid problem is chosen to be closest to the theoretical ones obtained from Theorem~\ref{thm:main} if not exactly the same.  As we can see from this figure, the two-grid solution is very close to the exact solution. Therefore, by the appropriately choice of the coarse problem, solving the nonlinear problem could be reduced to solving a linear problem on the fine mesh without loss of accuracy. Note that the linearized problem on the fine mesh could be solved efficiently by multilevel preconditioning techniques, even in the presence of large jump coefficients (cf. \cite{Xu.J;Zhu.Y2008}). In this way, we reduced greatly the overall computational cost for solving the nonlinear PDEs. 
\begin{figure}[htp]
	\begin{center}
		\includegraphics[width=0.7\linewidth]{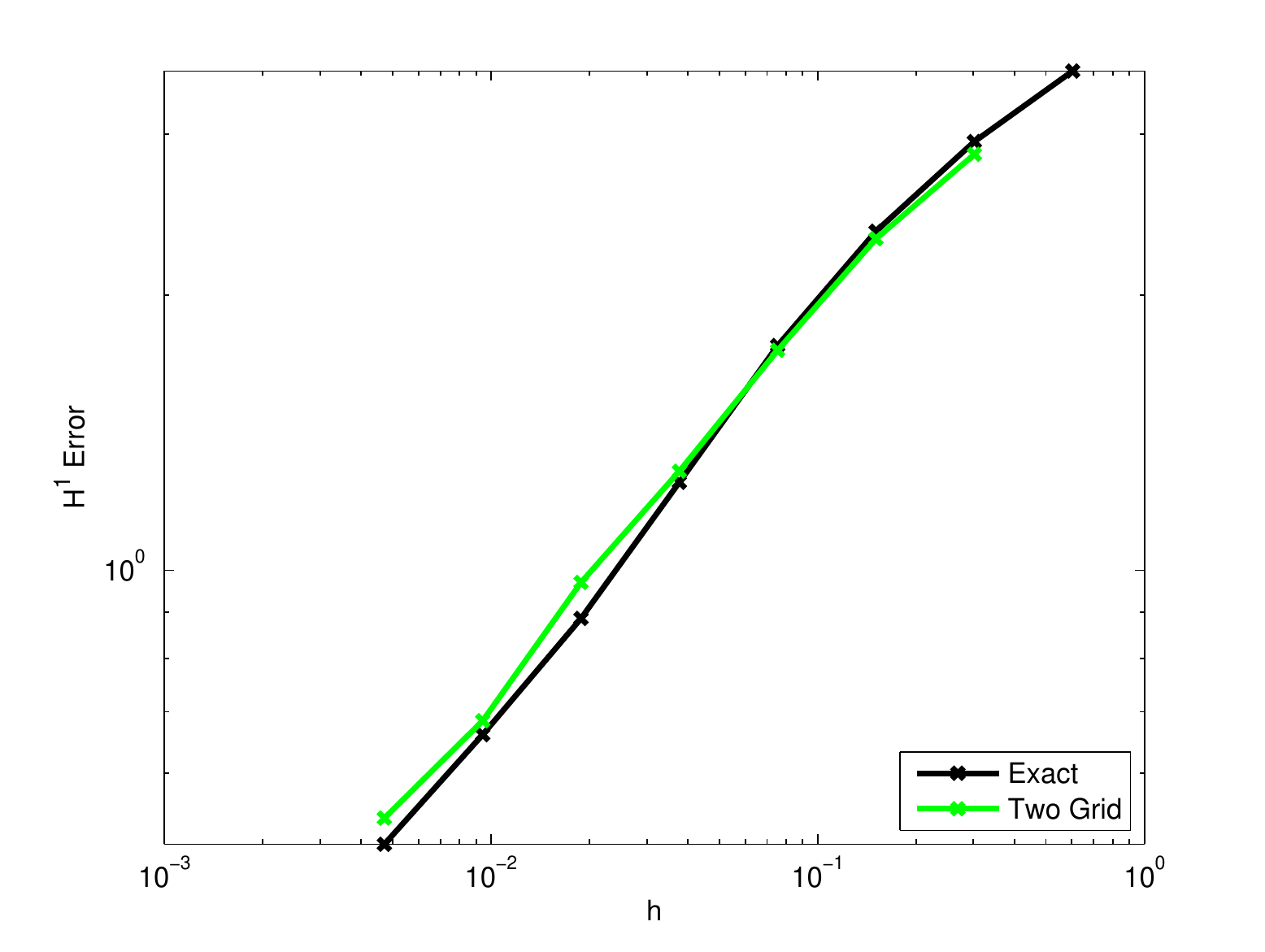} 
		\caption{Two-grid Error.}\label{fig:2grid}
	\end{center}
\end{figure}	

\section{Conclusion and Extension}
\label{sec:conc}

In this article we considered a two-grid finite element method for solving 
semilinear interface problems in $d$ space dimensions, for $d=2$ or $d=3$.
We first described in some detail the target problem class with 
discontinuous diffusion
coefficients, which included critical (and subcritical) 
nonlinearity examples, as well problems containing supercritical nonlinearity 
(such as the Poisson-Boltzmann equation and the semi-conductor device modeling
equations).
We then developed a basic quasi-optimal \emph{a priori} error estimate
for Galerkin approximations.
In the critical and subcritical cases, we follow~\cite{BHSZ11a}
and control the nonlinearity using only pointwise control of the 
continuous solution and a local Lipschitz 
property, rather than through pointwise control of the discrete solution;
this eliminates the requirement that the discrete solution satisfy a 
discrete form of the maximum principle, hence eliminating the need for 
restrictive angle conditions in the underlying mesh.
However, the supercritical case continues to require such conditions in
order to control the nonlinearity.
We then designed a two-grid algorithm consisting of a coarse grid solver 
for the original nonlinear problem, and a fine grid solver for a 
linearized problem. 
We analyzed the quality of approximations generated by the algorithm,
and proved that the coarse grid may be taken to have much larger elements
than the fine grid, and yet one can still obtain approximation quality that 
is asymptotically as good as solving the original nonlinear problem on the
fine mesh.
The included numerical experiments support our theoretical results.

The algorithm we described, and its analysis in this article, 
combined four sets of tools: 
the work of Xu and Zhou on two-grid algorithms for 
semilinear problems~\cite{Xu.J1996b,Xu.J;Zhou.A2000,Xu.J;Zhou.A2001a};
the recent results for linear interface problems due to 
Li, Melenk, Wohlmuth, and Zou~\cite{Li.J;Melenk.J;Wohlmuth.B;Zou.J2010};
recent work on the Poisson-Boltzmann equation~\cite{CHX06b,Holst.M;McCammon.J;Yu.Z;Zhou.Y2012};
and recent results on \emph{a priori} estimates for semilinear problems,
including estimates without angle conditions in the case of
sub- and super-critical nonlinearity~\cite{BHSZ11a}.
Although the algorithm described in this paper is applicable to general
coupled nonlinear elliptic systems, our reliance on tools developed for 
scalar linear and semilinear problems restricts the validity of the
theoretical results to the class of semilinear problems described
in~\S\ref{sec:pde}.
In future work we will consider the case of coupled systems of scalar
semilinear PDE from this class, as well as more general nonlinear
elliptic systems.

To simplify the presentation and keep the paper focused, we assumed that the triangulations resolve the interface. For general interface $\Gamma$, namely, $\Gamma$ can not be resolved by the triangulation, we could use the concept of ``$\delta$-resolved triangulation'' (cf. \cite{Li.J;Melenk.J;Wohlmuth.B;Zou.J2010}).  The results in this article could be generalized in a direct way if the triangulation satisfies the ``$\delta$-resolved''. 
Without significant technical modifications to the results in the article, we could also relax the Local Monotonicity Assumption~\ref{ass:bm} to the following:
$$b'(\xi)> -\lambda_1,$$
where $\lambda_1$ is the smallest eigenvalue of the 
operator $-\nabla\cdot (D\nabla \cdot)$.

\section{Acknowledgments}
   \label{sec:ack}

MH was supported in part by NSF Awards~0715146 and 0915220,
and by DOD/DTRA Award HDTRA-09-1-0036.
RS and YZ were supported in part by NSF Award~0715146.

\bibliographystyle{abbrv}
\bibliography{../bib/books,../bib/papers,../bib/mjh,../bib/library,../bib/ref-gn,../bib/coupling,../bib/pnp}

\begin{thebibliography}{10}

\bibitem{Antal.I;Karatson.J2008}
I.~Antal and J.~Kar{\'a}tson.
\newblock {Mesh independent superlinear convergence of an inner--outer
  iterative method for semilinear elliptic interface problems}.
\newblock {\em Journal of Computational and Applied Mathematics}, 2008.

\bibitem{Axelsson.O;Layton.W1996}
O.~Axelsson and W.~Layton.
\newblock A two-level method for the discretization of nonlinear boundary value
  problems.
\newblock {\em SIAM journal on numerical analysis}, 33(6):2359--2374, 1996.

\bibitem{Babuska.I1970}
I.~Babu{\v s}ka.
\newblock The finite element method for elliptic equations with discontinuous
  coefficients.
\newblock {\em Computing}, 5(3):207--213, 1970.

\bibitem{BHSZ11a}
R.~Bank, M.~Holst, R.~Szypowski, and Y.~Zhu.
\newblock Finite element error estimates for critical exponent semilinear
  problems without angle conditions.
\newblock Submitted for publication. Available as
  \href{http://arxiv.org/abs/1108.3661} {{\sf arXiv:1108.3661 [math.NA]}}.

\bibitem{Bank.R;Rose.D1981}
R.~E. Bank and D.~J. Rose.
\newblock Global approximate {N}ewton methods.
\newblock {\em Numerische Mathematik}, 37:279--295, 1981.

\bibitem{Bank.R;Rose.D1982}
R.~E. Bank and D.~J. Rose.
\newblock Analysis of a multilevel iterative method for nonlinear finite
  element equations.
\newblock {\em Mathematics of Computation}, 39:453--465, 1982.

\bibitem{Barrett.J;Elliott.C1987}
J.~Barrett and C.~Elliott.
\newblock Fitted and unfitted finite-element methods for elliptic equations
  with smooth interfaces.
\newblock {\em IMA journal of numerical analysis}, 7(3):283--300, 1987.

\bibitem{Brezzi.F;Rappaz.J;Raviart.P1980}
F.~Brezzi, J.~Rappaz, and P.~A. Raviart.
\newblock Finite dimensional approximation of nonlinear problems part i:
  Branches of nonsingular solutions.
\newblock {\em Numer. Math.}, 36:1--25, 1980.

\bibitem{Chen.L;Holst.M;Xu.J2007}
L.~Chen, M.~Holst, and J.~Xu.
\newblock The finite element approximation of the nonlinear {Poisson-Boltzmann}
  equation.
\newblock {\em SIAM Journal on Numerical Analysis}, 45(6):2298--2320, 2007.

\bibitem{CHX06b}
L.~Chen, M.~Holst, and J.~Xu.
\newblock The finite element approximation of the nonlinear {Poisson-Boltzmann}
  {Equation}.
\newblock {\em SIAM J.\ Numer.\ Anal.}, 45(6):2298--2320, 2007.
\newblock Available as \href{http://arxiv.org/abs/1001.1350} {{\sf
  arXiv:1001.1350 [math.NA]}}.

\bibitem{Chen.Z;Zou.J1998}
Z.~Chen and J.~Zou.
\newblock Finite element methods and their convergence for elliptic and
  parabolic interface problems.
\newblock {\em Numerische Mathematik}, 79(2):175--202, 1998.

\bibitem{Ciarlet.P;Raviart.P1973}
P.~G. Ciarlet and P.~A. Raviart.
\newblock Maximum principle and uniform convergence for the finite element
  method.
\newblock {\em Computer Methods in Applied Mechanics and Engineering},
  2:17--31, 1973.

\bibitem{Hansbo.P;Lovadina.C;Perugia.I;Sangalli.G2005}
P.~Hansbo, C.~Lovadina, I.~Perugia, and G.~Sangalli.
\newblock A {L}agrange multiplier method for the finite element solution of
  elliptic interface problems using non-matching meshes.
\newblock {\em Numer. Math.}, 100(1):91--115, 2005.

\bibitem{Holst.M;McCammon.J;Yu.Z;Zhou.Y2012}
M.~Holst, J.~McCammon, Z.~Yu, Y.~Zhou, and Y.~Zhu.
\newblock Adaptive finite element modeling techniques for the
  {Poisson-Boltzmann} equation.
\newblock {\em Communications in Computational Physics}, 11(1):179--214, 2012.
\newblock Available as \href{http://arxiv.org/abs/1009.6034} {{\sf
  arXiv:1009.6034 [math.NA]}}.

\bibitem{Jungel.A;Unterreiter.A2005}
A.~J{\"u}ngel and A.~Unterreiter.
\newblock Discrete minimum and maximum principles for finite element
  approximations of non-monotone elliptic equations.
\newblock {\em Numer. Math.}, 99(3):485--508, 2005.

\bibitem{Karatson.J;Korotov.S2006}
J.~Kar\'{a}tson and S.~Korotov.
\newblock Discrete maximum principles for fem solutions of some nonlinear
  elliptic interface problems.
\newblock Research Reprots A510, Helsinki University of Technology, Institute
  of Mathematics, 2006.

\bibitem{Kellogg.R1975}
R.~B. Kellogg.
\newblock On the {Poisson} equation with intersecting interface.
\newblock {\em Appl. Aanal.}, 4:101--129, 1975.

\bibitem{Kerkhoven.T;Jerome.J1990}
T.~Kerkhoven and J.~W. Jerome.
\newblock {$L_{\infty}$} stability of finite element approximations of elliptic
  gradient equations.
\newblock {\em Numerische Mathematik}, 57:561--575, 1990.

\bibitem{Ladyzhenskaya.O1969}
O.~A. Ladyzhenskaya.
\newblock {\em The Mathematical Theory Of Viscous Incompressible Fluid}.
\newblock Gordon and Breach, 1969.

\bibitem{Lamichhane.B;Wohlmuth.B2004}
B.~P. Lamichhane and B.~I. Wohlmuth.
\newblock Mortar finite elements for interface problems.
\newblock {\em Computing}, 72(3 - 4):333--348, 2004.

\bibitem{Li.J;Melenk.J;Wohlmuth.B;Zou.J2010}
J.~Li, J.~Melenk, B.~Wohlmuth, and J.~Zou.
\newblock {Optimal a priori estimates for higher order finite elements for
  elliptic interface problems}.
\newblock {\em Applied numerical mathematics}, 60(1-2):19--37, 2010.

\bibitem{Li.Z;Lin.T;Wu.X2003}
Z.~Li, T.~Lin, and X.~Wu.
\newblock New cartesian grid methods for interface problems using the finite
  element formulation.
\newblock {\em Numerische Mathematik}, 96(1):61--98, 2003.

\bibitem{Nicaise.S;Sandig.A1994}
S.~Nicaise and A.~S{\"a}ndig.
\newblock {General Interface Problems-I}.
\newblock {\em Mathematical Methods in the Applied Sciences}, 17(6):395--429,
  1994.

\bibitem{Nicaise.S;Sandig.A1994a}
S.~Nicaise and A.~S{\"a}ndig.
\newblock {General Interface Problems-II}.
\newblock {\em Math. Methods Appl. Sci}, 17:431--450, 1994.

\bibitem{Plum.M;Wieners.C2003}
M.~Plum and C.~Wieners.
\newblock Optimal a priori estimates for interface problems.
\newblock {\em Numer. Math.}, 95:735--759, 2003.

\bibitem{Rannacher.R1991}
R.~Rannacher.
\newblock On the convergence of the {Newton-Raphson} method for strongly
  nonlinear finite element equations.
\newblock In P.~Wriggers and W.~Wanger, editors, {\em Nonlinear Computational
  Mechanics}. Springer Berlin / Heidelberg, 1991.

\bibitem{Sinha.R;Deka.B2006}
R.~K. Sinha and B.~Deka.
\newblock On the convergence of finite element method for second order elliptic
  interface problems.
\newblock {\em Numer. Funct. Anal. Optim.}, 27(1):99--115, 2006.

\bibitem{Sinha.R;Deka.B2009}
R.~K. Sinha and B.~Deka.
\newblock Finite element methods for semilinear elliptic and parabolic
  interface problems.
\newblock {\em Applied Numerical Mathematics}, 59(8):1870 -- 1883, 2009.

\bibitem{StHo2011a}
I.~Stakgold and M.~Holst.
\newblock {\em Green's Functions and Boundary Value Problems}.
\newblock John Wiley \& Sons, Inc., New York, NY, third edition, 888 pages,
  February 2011.
\newblock The preface and table of contents of the book are available at:
  \href{http://ccom.ucsd.edu/~mholst/pubs/dist/StHo2011a-preview.pdf} {{\sf
  http://ccom.ucsd.edu/\~{}mholst/pubs/dist/StHo2011a-preview.pdf}}.

\bibitem{Wang.J;Zhang.R2011}
J.~Wang and R.~Zhang.
\newblock Maximum principles for {$P1$}-conforming finite element
  approximations of quasi-linear second order elliptic equations.
\newblock {\em Arxiv preprint arXiv:1105:1466}, 2011.

\bibitem{Xu.J1982}
J.~Xu.
\newblock Error estimates of the finite element method for the 2nd order
  elliptic equation with discontinuous coefficient.
\newblock {\em J. Xiangtan Univ}, (1), 1982.

\bibitem{Xu.J1996b}
J.~Xu.
\newblock Two-grid discretization techniques for linear and nonlinear
  {P}{D}{E}s.
\newblock {\em SIAM Journal on Numerical Analysis}, 33(5):1759--1777, 1996.

\bibitem{Xu.J;Zhou.A2000}
J.~Xu and A.~Zhou.
\newblock Local and parallel finite element algorithms based on two-grid
  discretizations.
\newblock {\em Mathematics of Computation}, 231:881--909, 2000.

\bibitem{Xu.J;Zhou.A2001a}
J.~Xu and A.~Zhou.
\newblock Local and parallel finite element algorithms based on two-grid
  discretizations for nonlinear problems.
\newblock {\em Advances in Comp. Math.}, 14(4):293--327, 2001.

\bibitem{Xu.J;Zhu.Y2008}
J.~Xu and Y.~Zhu.
\newblock Uniform convergent multigrid methods for elliptic problems with
  strongly discontinuous coefficients.
\newblock {\em Mathematical Models and Methods in Applied Science}, 18(1):77
  --105, 2008.

\bibitem{Zhu.Y2008}
Y.~Zhu.
\newblock Domain decomposition preconditioners for elliptic equations with jump
  coefficients.
\newblock {\em Numerical Linear Algebra with Applications}, 15(2-3):271--289,
  2008.

\end{thebibliography}

\vspace*{0.5cm}

\end{document}